\documentclass[11pt]{amsart}
\usepackage{amsmath}
\usepackage{amssymb}
\usepackage{amsfonts}
\usepackage{latexsym}
\usepackage{fancyhdr}
\usepackage{graphics}
\usepackage{color}
\newcommand{\mb}{\mathbb}
\newcommand{\mc}{\mathcal}

\newcommand{\M}{\mc M}

\newcommand{\D}{\mc D}

\def\H{\mc H}

\newcommand{\SC}{\mathcal{S}({\mb R})}
\newcommand{\SSS}{\mathcal{S}^\times({\mb R})}

\newcommand{\LDD}{{\mc L}(\D,\D^\times)}
\newtheorem{defn}{Definition}[section]
\newtheorem{prop}[defn]{Proposition}

\newtheorem{lemma}[defn]{Lemma}
\newtheorem{cor}[defn]{Corollary}
\newtheorem{example}[defn]{Example}
\newtheorem{rem}[defn]{Remark}
\def\x{\relax\ifmmode {\mbox{*}}\else*\fi}
\newcommand{\beex}{\begin{example}$\!\!${\bf }$\;$\rm }
\newcommand{\enex}{ \end{example}}
\newcommand{\berem}{\begin{rem}$\!\!${\bf }$\;$\rm }
\newcommand{\enrem}{ \end{rem}}

\newcommand{\bedefi}{\begin{defn}$\!\!${\bf }$\;$\rm }
\newcommand{\findefi}{\end{defn}}

\newcommand{\ip}[2]{\left\langle {#1}|{#2}\right\rangle}
\newcommand{\RN}{{\mb R}}

\begin{document}

\title[Distribution Riesz-Fischer]{ Riesz-Fischer maps, Semiframes and Frames\\ in rigged Hilbert spaces}%

\author{Francesco Tschinke}%
\address{Dipartimento di Matematica e Informatica, Universit\`a
di Palermo, I-90123 Palermo (Italy)} \email{francesco.tschinke@unipa.it}

\begin{abstract}
\noindent
In this note a review, some  considerations and new results about maps with values in a distribution space and domain in a $\sigma$-finite measure space $X$,  are obtained. In particular, it is a survey about Bessel, frames and bases (in particular Riesz and Gel'fand bases) in a distribution space. In this setting, the Riesz-Fischer maps and semi-frames are defined and new results about them are attained. Some example in tempered distributions space are examined.

\end{abstract}
\maketitle
\section{Introduction}
 Given a Hilbert space $\H$ with inner product $\ip{\cdot}{\cdot}$ and norm $\|\cdot\|$, a frame is a sequence of   vectors  $\{f_n\}$  in $\H$  if  there exists $ A,B>0$:
$$
 A\|f\|^2\leq\sum_{k=1}^\infty|{\ip{f}{f_n}}|^2\leq B\|f\|^2,\quad\forall f\in\H.
$$
As known, this notion generalize  orthonormal bases, and  has reached a crescent level of popularity in many fields of interests, such  signal theory, image processing, etc., but it is also important tool in pure mathematics: in fact it  plays key roles in wavelets theory, time-frequency analysis, the theory of shiftinvariant
spaces, sampling theory and many other areas (see \cite{christensen}\cite{christensen2}\cite{groch}\cite{heil}\cite{keiser}).

A generalization of frame,  the {\it continuous} frame, was proposed by Kaiser, and by Al\'{\i}, Antoine, Gazeau: if $(X,\mu)$ is a measure space  with $\mu$ as $\sigma$-finite positive measure, a  {\it weakly} measurable function $f_x:X\rightarrow\H$ ( i.e.:  $x\mapsto\ip{f}{f_x}$ is $\mu$-measurable on  $X$) is a {\it continuous} frame of $\H$ w.r. to $\mu$  if there exists $A,B>0$ s.t., for all $g\in\H$:
$$
A\|g\|^2\leq\int_X|\ip{g}{f_x}|^2 d\mu\leq B\|g\|^2,\quad\forall g\in\H.
$$
Today, the notion of continuous frames in Hilbert spaces and its link with the theory of coherent states is well-known in the literature \cite{AAG_book, AAG_paper}, \cite{keiser}.

With the collaboration of C. Trapani and T. Triolo \cite{TTT},  the author   introduces bases and frames in a space of distributions. To illustrate the motivations for this study, we have to consider the
  {\it rigged Hilbert space} (or Gel'fand triple) \cite{gelf3,gelf}: that is, if $\H$ is a Hilbert space, the triple:
 $$
 \D[t]\subset\H\subset\D^\times[t^\times].
 $$
where $\D[t]$ is a dense subspace of $\H$ endowed with a locally convex topology $t$ stronger than the Hilbert norm and $\D^\times[t^\times]$ is the conjugate dual space of $\D$ endowed with strong dual topology $t^\times$ (the inclusions are dense and continuous).

In this setting, let us consider the \textit{generalized eigenvectors} of an operator, i.e. eigenvectors  that do not belong to $\H$. More precisely:
 if $A$ is an essentially self-adjoint operator in $\D$ which maps $\D[t]$ into $\D[t]$ continuously, then $A$ has a continuous extension $\hat{A}$   given by the adjoint, (i.e. $\hat{A}=A^\dag$) from  $\D^\times$ into itself.
A \textit{generalized eigenvector} of $A$, with eigenvalue $\lambda\in\mathbb C$, is  an eigenvector of $\hat{A}$; that is, a conjugate linear functional $\omega_\lambda\in\D^\times$ such that:
$$
\ip{Af}{\omega_\lambda}=\lambda \ip{f}{\omega_\lambda},\quad\forall f\in\D.
$$
The above equality can be read as $\hat{A}\omega_\lambda=A^\dag\omega_\lambda=\lambda\omega_\lambda$.

 A simple and explicative example is given by the derivative operator: $A:=i\frac{d}{dx}:\mathcal S(\mathbb R)\rightarrow S(\mathbb R)$ where $\mathcal S(\mathbb R)$ is the Schwartz space (i.e. infinitely differentiable functions rapidely decreasing). The rigged Hilbert space is:
\begin{equation}\label{schwartz}
\mathcal S(\mathbb R)\subset  L^2(\mathbb R)\subset\mathcal S^\times(\mathbb R),
\end{equation}
the set $\mathcal S^\times(\mathbb R)$ is known as space of {\it tempered distributions}. Then $\omega_\lambda(x)=\frac{1}{\sqrt{2\pi}}e^{-i\lambda x}$ -that do not belongs to $L^2(\mathbb R)$-  is a generalized eigenvector of $A$ with $\lambda$ as eigenvalue.


Each function $\omega_\lambda$ can be viewed as a regular distribution of $\SSS$, in the sense that, for every fixed $\lambda\in {\mb R}$, through  the following integral representation:
$$  \ip{\phi}{\omega_\lambda} =\int_\RN \phi(x)\overline{\omega_\lambda(x)}dx =\frac{1}{\sqrt{2\pi}}\int_\RN \phi(x)e^{i\lambda x}dx=\check{\phi}(\lambda)
$$
it is defined a continuous linear functional $\phi\mapsto\check{\phi}(\lambda)$ on $\SC$. Furthermore
for the Fourier-Plancherel theorem, one has: $\|\check{\phi}\|^2_2=\int_\RN|\ip{\phi}{\omega_\lambda}|^2dx=\|\phi\|^2_2$.

With a limit procedure, Fourier transform can be extended to $L^2(\mathbb{R})$. Since a function $f\in L^2(\mathbb{R})$ defines a regular tempered distribution, we have, for all $\phi\in\mathcal S(\mathbb R)$:
$$\ip{\phi}{f}\!:=\!\!\int_{\mathbb R}\!\!\!f(x)\phi(x)dx\!=\!\!\int_{\mathbb R}\!\!\left(\!\!\frac{1}{\sqrt{2\pi}}\!\int_\RN \!\!\hat{f}(\lambda)e^{i\lambda x}d\lambda\! \right)\! \phi(x)dx\!=\!\!\!\int_{\mathbb R}\!\hat{f}(\lambda)\check{\phi}(\lambda)d\lambda\!=\!\!\!\int_{\mathbb R}\!\!\!\hat{f}(\lambda)\!\ip{\phi}{\omega_\lambda}d\lambda.
$$
That is:
\begin{equation}
\label{eqn_fourier}
f=\int_{\mathbb R}\!\hat{f}(\lambda)\omega_\lambda d\lambda.
\end{equation}
in weak sense.
The family $\{\omega_\lambda;\,\lambda \in {\mb R}\}$ of the previous example can be considered as the range of a weakly measurable function $\omega: {\mb R}\to \SSS$ which allows a representation \eqref{eqn_fourier}  of any $f\in L^2(\mathbb R)$ in terms of generalized eigenvectors of $A$: it is an example of {\it distribution basis}. More precisely, since the Fourier-Plancherel theorem corresponds to the Parseval identity, this is an example of {\em Gel'fand distribution bases}, that plays, in $\mathcal S^\times(\RN)$, the role  of orthonormal basis in Hilbert space.

The example above is a particular case of the Gel'fand-Maurin theorem  (see \cite{gelf}\cite{gould} for details), that  stats that, if $\D$ is a domain in a Hilbert space $\H$ which is a nuclear space under a certain topology $\tau$, and $A$ is an essentially self-adjoint operator in $\D$ which maps $\D[t]$ into $\D[t]$ continuously, then $A$ admits a {\em complete  set of generalized eigenvectors}.

The completeness of the set $\{\omega_\lambda; \lambda \in \sigma(\overline{A})$\} is understood in the sense that  the Parseval identity holds, that is:
\begin{equation}
\label{gelfand}
\|f\|= \left(\int_{\sigma(\overline{A})} |\ip{f}{\omega_\lambda}|^2 d\lambda\right)^{1/2}, \quad \forall f\in \D.
\end{equation}
To each $\lambda$ there corresponds the subspace $\D^\times_\lambda \subset \D^\times$ of all  generalized eigenvectors whose eigenvalue is $\lambda$. For all $f\in\D$ it is possible to define a linear functional $\tilde{f}_\lambda$ on $\D^\times_\lambda$ by $\tilde{f}_\lambda(\omega_\lambda):=\ip{\omega_\lambda}{f}$ for all $\omega_\lambda\in\D_\lambda^\times$. The correspondence $\D\rightarrow\D_\lambda^{\times\times}$ defined by $f\mapsto \tilde{f}_\lambda$ is called the \textit{spectral decomposition of the element $f$ corresponding to $A$}. If $\tilde{f}_\lambda\equiv 0$ implies $f=0$ (i.e. the map $f\mapsto \tilde{f}_\lambda$ is injective) then $A$ is said to have \textit{a complete system of generalized eigenvectors}.

The completeness and condition (\ref{gelfand})  gives also account of a kind of {\em orthogonality} of the $\omega_\lambda$'s: the family $\{\omega_\lambda\}_{\lambda\in\sigma(\overline{A})}$  in  \cite{TTT} is called {\em Gel'fand basis} .

Another meaningful situation come from quantum mechanics, known as  the \textit{spectral expansion theorem}.  Let us consider    the rigged Hilbert space \eqref{schwartz} (one-dimensional case).
The  Hamiltonian operator $H$ is  an essentially self-adjoint operator on $\mathcal S(\mathbb R)$,  with self-adjoint extension $\overline{H}$ on the domain $\D(\overline H)$, dense in $L^2(\mathbb R)$.
Then, in the case of non-degenerate spectrum,  for  all $f\in L^2(\mathbb R)$ the following decomposition holds:
$$
f=\sum_{n\in J} c_n u_n+\int_{\sigma_c} c(\alpha)u_\alpha d\mu(\alpha).
$$
The set  $\{u_n\}_{n\in J}$, $J\subset \mathbb N$,  is an orthonormal system of eigenvectors of $H$;
 the measure $\mu$ is a continuous measure on $\sigma_c\subset\mathbb R$ and
 $\{u_\alpha\}_{\alpha\in{\sigma_c}}$ are {\it generalized eigenvectors} of $H$ in $\mathcal S^\times(\mathbb R)$. This decomposition is {\it unique}.
 Furthermore:
$$
\|f\|^2=\sum_{n\in J}|c_n|^2+\int_{\sigma_c}|c(\alpha)|^2d\mu(\alpha).
$$
The subset $\sigma_c$, corresponding to the continuous spectrum,  is a union of intervals of $\mathbb R$, i.e. the index $\alpha$  is continuous. The generalized eigenvectors $u_\alpha$ are distributions: does not belongs to $ L^2(\mathbb R)$, therefore the ``orthonormality'' between generalized eigenvectors is not defined. Nevertheless, it is  often denoted by the  physicists with the Dirac delta: $``\ip{u_\alpha}{u_{\alpha'}}$''=$\delta_{\alpha-\alpha'}$.

Frames, semi-frames, Riesz bases,  etc. are families of vectors that generalize bases in Hilbert space maintaining the possibility to {\em reconstruct} vectors of the space as the superposition of more 'elementary' vectors
renouncing often to the uniqueness of the representation, but gaining in  versatility.

In this sense, they have been considered in literature in various space of functions and distributions: see for example (but the list  is not exhaustive): \cite{jpa_ct_rppip, gb_ct_riesz,cordero,feich2,feich 3, Kpet,stoevapilipovic}.

It is remarkable that in a separable Hilbert space, orthonormal bases as well as Riesz bases are both necessarily countable and also in more general situations Riesz bases cannot be {\em continuous}, but they are {\em discrete} \cite{hosseini, jacobsen, bal_ms_2}. In the distributions and rigged Hilbert space setting the corresponding objects can be continuous, although  the families of functions forming the frame are less regular.

Revisiting some results of  \cite{TTT} about Bessel map, frames and bases (Gel'fand and Riesz bases) in distribution set-up,  in this paper  it is proposed  the notion of Riesz-Fischer map and of semi-frame in a space of distributions.

After some preliminaries and notations (Section \ref{sect_2}), in Section \ref{sect_3} are considered distribution Bessel maps and it is proposed the notion of distribution Riesz-Fischer map, showing some new results about them (such as bounds and duality properties). Since distribution Bessel maps are not, in general, bounded by an Hilbert norm,  the author consider appropriate to define in Section \ref{sect_4} the {\em distribution semi-frames}, notion already introduced in a Hilbert space by J.-P. Antoine and P. Balasz \cite{semifr1}. Finally, distribution frames,  distribution basis,  Gel'fand and Riesz basis, considered in \cite{TTT}, are revisited in Section \ref{sect_5} with some examples in addition.
\medskip
\section{Preliminary definitions and facts}
\label{sect_2}
\subsection{Rigged Hilbert space}
  Let $\D$ be a dense subspace of  $\H$ endowed with a locally convex topology $t$  finer than the topology induced by the Hilbert norm.
 Denote $\D^\times$  as the vector space of all continuous conjugate
linear functionals on $\D[t]$, i.e., the conjugate dual of $\D[t]$,
endowed with the {\em strong dual topology} $t^\times=
\beta(\D^\times,\D)$, which can be defined by the seminorms
\begin{equation}\label{semin_Dtimes}
q_\M(F)=\sup_{g\in \M}|\ip{F}{g}|, \quad F\in \D^\times,
\end{equation}
where $\M$ is a bounded subsets of $\D[t]$. In this way, it is defined, in standard fashion,
a {\em rigged Hilbert space}:
\begin{equation}\label{eq_one_intr}
\D[t] \hookrightarrow  \H \hookrightarrow\D^\times[t^\times],
\end{equation}
Where $\hookrightarrow$ denote continuous and dense injection.

Since the Hilbert space $\H$ can be identified  with a
subspace of $\D^\times[t^\times]$, we will systematically read
\eqref{eq_one_intr} as a chain of topological inclusions: $\D[t]
\subset  \H \subset\D^\times[t^\times]$.  These identifications
imply that the sesquilinear form $B( \cdot , \cdot )$ which puts $\D$
and $\D^\times$ in duality is an extension of the inner product of
$\H$;
 i.e. $B(\xi, \eta) = \ip{\xi}{\eta}$, for every $\xi, \eta \in \D$ (to simplify notations we adopt the symbol $\ip{\cdot}{\cdot}$ for both of
 them) and also that the embedding map $I_{\D,\D^\times}:\D\to \D^\times$ can be taken to act on $\D$ as $I_{\D,\D^\times}f=f$ for every $f \in \D$. For more  insights, besides \cite{gelf3,gelf}, see also \cite{horvath}.
\subsection{$\LDD$-space}
If $\D[t] \subset \H \subset \D^\times[t^\times]$ is a rigged
Hilbert space, let us   denote $\LDD$ as the vector space of all continuous linear maps from $\D[t]$ into  $\D^\times[t^\times]$. {If $\D[t]$ is barreled (e.g., reflexive)}, an involution $X \mapsto X^\dag$ can be introduced in $\LDD$  by the identity:
$$ \ip{X^\dag \eta}{ \xi} = \overline{\ip{X\xi}{\eta}}, \quad \forall \xi, \eta \in \D.$$  Hence, in this case, $\LDD$ is a $^\dagger$-invariant vector space.


We also denote by ${\mc L}(\D)$ the algebra of all continuous linear operators $Y:\D[t]\to \D[t]$ and by ${\mc L}(\D^\times)$ the algebra of all continuous linear operators $Z:\D^\times[t^\times]\to \D^\times[t^\times]$.
If $\D[t]$ is reflexive, for every $Y \in {\mc L}(\D)$ there exists a unique operator $Y^\times\in {\mc L}(\D^\times)$, the adjoint of $Y$, such that
$$ \ip{\Phi}{Yg} = \ip{Y^\times \Phi}{g}, \quad\forall \Phi \in \D^\times, g \in \D.$$ In similar way an operator $Z\in {\mc L}(\D^\times)$ has an adjoint $Z^\times\in {\mc L}(\D)$ such that $(Z^\times)^\times=Z$. In the monograph \cite{ait_book} the topic is treated more deeply.

\subsection{Weakly measurable maps}
In this paper is considered the {\it weakly measurable map} as a subset of $\D^\times$: if $(X,\mu)$ is a measure space with $\mu$ a $\sigma$-finite positive measure,
  $\omega: x\in X\to \omega_x \in \D^\times$ is a weakly measurable map if, for every $f \in \D$, the complex valued function $x \mapsto \ip{f}{\omega_x}$ is $\mu$-measurable.
Since the form which puts $\D$ and $\D^\times$ in conjugate duality is an extension of the inner product of $\H$, we write $\ip{f}{\omega_x}$ for $\overline{\ip{\omega_x}{f}}$,  $f \in \D$. We have the following:
\bedefi
\label{tandg}
 Let $\D[t]$ be a locally convex space, $\D^\times$ its conjugate dual and $\omega: x\in X\to \omega_x \in \D^\times$ a weakly measurable map, then:
\begin{itemize}
\item[i)] $\omega$ is \textit{total} if, $f \in \D $ and $\ip{f}{\omega_x}=0$  $\mu$-a.e. $x \in X$ implies $f=0$;
\item[ii)]$\omega$ is \textit{$\mu$-independent} if the unique measurable function $\xi$ on $(X,\mu)$ such that: \\$\int_X \xi(x)\ip{g}{\omega_x} d\mu=0$, for every $g \in \D$, is $\xi(x)=0$ $\mu$-a.e.
\end{itemize}
\findefi
\section{Bessel and Risz-Fischer distribution maps}\label{sect_3}
\subsection{Bessel distribution maps on locally convex spaces}
\bedefi Let $\D[t]$ be a locally convex space. A weakly measurable map $\omega$ is a {\em Bessel distribution map} (briefly: Bessel map) if for every $f \in \D$,
$ \int_X |\ip{f}{\omega_x}|^2d\mu<\infty$.
\findefi
The following Proposition, proved in \cite{TTT}, is the analogue of Proposition 2 and Theorem 3 in \cite[sec.2, Cap.4]{young}.
\begin{prop}\label{prop2} If $\D[t]$ a Fr\'{e}chet space, and $\omega: x\in X \to \omega_x\in \D^\times$ a weakly measurable map. The following statements are equivalent.
\begin{itemize}
\item[(i)]
$\omega$ is a  Bessel  map;
\item[(ii)]there exists a continuous seminorm $p$ on $\D[t]$ such that:
\begin{equation}\label{eqn_bessel1}\left( \int_X |\ip{f}{\omega_x}|^2d\mu\right)^{1/2}\leq p(f), \quad \forall f \in \D.\end{equation}
 \item[(iii)] for every bounded subset $\mathcal M$ of $\D$ there exists $C_{\mathcal M}>0$ such that:
\begin{equation}
\label{disbessel}
\sup_{f\in\mathcal M}{\Bigr |}\int_X\xi(x)\ip{\omega_x}{f}d\mu{\Bigl |}\leq C_{\mathcal M}\|\xi\|_2, \quad \forall \xi\in L^2(X,\mu).
\end{equation}
  \end{itemize}
\end{prop}
We have also the following  \cite{TTT}:
\begin{lemma}
If $\D$ is a Fr\'echet space and let $\omega$ be a Bessel distribution map. Then:
$$
\int_X\ip{f}{\omega_x}\omega_xd\mu
$$
converges  for every $f\in\D$ to an element of $\D^\times$.
Moreover, the map $\D\ni f \mapsto \int_X\ip{f}{\omega_x}\omega_xd\mu \in \D^\times$ is continuous.
\end{lemma}

{
Let $\omega$ be a Bessel map: the previous lemma allows to define on $\D\times \D$   the sesquilinear form $\Omega$:
\begin{equation}\label{eqn_omega} \Omega(f,g)= \int_X \ip{f}{\omega_x}\ip{\omega_x}{g}d\mu.
\end{equation}
By Proposition \ref{prop2}, for some continuous seminorm on $\D$, one has:
$$ |\Omega(f,g)|= \left| \int_X \ip{f}{\omega_x}\ip{g}{\omega_x}d\mu\right|\leq \|\ip{f}{\omega_x}\|_2\|\ip{g}{\omega_x}\|_2\leq p(f)p(g), \quad \forall f,g \in \D$$
for some continuous seminorm $p$ on $\D[t]$.}
This means that $\Omega$ is jointly continuous on $\D[t]$. Hence there exists an operator $S_\omega \in \LDD$, with $S_\omega =S_\omega^\dag$, $S_\omega\geq 0$, such that:
\begin{equation} \label{eqn_frameop}\Omega(f,g)=\ip{S_\omega f}{g}=\int_X \ip{f}{\omega_x}\ip{\omega_x}{g} d\mu, \quad \forall f,g\in \D\end{equation}
that is,
$$S_\omega f= \int_X \ip{f}{\omega_x}\omega_x d\mu, \quad \forall f \in \D.$$
Since $\omega$ is a Bessel distribution map and $\xi\in L^2(X, \mu)$, we put for all $g\in\D$:
\begin{equation}\label{eqn_lambda_om}\Lambda_\omega^\xi(g):= \int_X \xi(x)\ip{\omega_x}{g}d\mu .\end{equation}
Then $\Lambda_\omega^\xi$ is a continuous conjugate linear functional on $\D$, i.e.  $\Lambda_\omega^\xi\in \D^\times$. We write:
$$ {\Lambda^\xi_\omega}:=\int_X\xi(x){\omega_x}d\mu$$
in weak sense. Therefore we can define a linear map $T_\omega:L^2(X, \mu)\to \D^\times$, which will be called the {\em synthesis operator}, by:
$$ T_\omega: \xi \mapsto {\Lambda^\xi_\omega}.$$
By \eqref{disbessel}, it follows that $T_\omega$ is continuous from $L^2(X, \mu)$, endowed with its natural norm, into $\D^\times[t^\times]$. Hence, it possess a continuous adjoint $T_\omega^\times: \D[t]\to L^2(X, \mu)$, which is called the {\em analysis operator}, acting as follows:
$$T_\omega^\times: f\in \D \to \xi_f\in L^2(X, \mu), \mbox{ where } \xi_f(x)=\ip{f}{\omega_x}, \; x\in X.$$
One has that $S_\omega=T_\omega T_\omega^\times.$

\subsection{Riesz-Fischer distribution map}
\bedefi
Let $\D[t]$ be a locally convex space. A weakly measurable map $\omega: x\in X \to \omega_x\in \D^\times$ is called a {\em Riesz-Fischer distribution map} (briefly: Riesz-Fischer map)  if,
for every $h\in L^2({X,\mu})$,
there exists $f \in \D$ such that:
\begin{equation}
\label{rf}
\ip{f}{\omega_x}=h(x)\quad \mu-a.e.
\end{equation}
In this case, we say  that $f$ is a solution of equation $\ip{f}{\omega_x}=h(x)$.
\findefi
Clearly, if $f_1$ and $f_2$ are solutions of (\ref{rf}), then $f_1-f_2\in\omega^\bot:=\{g\in \D: \ip{g}{\omega_x}=0, \quad \mu-a.e.\}$. If $\omega$ is total, the solution is unique.
We prove the following:
\begin{lemma}
Let $\D$ be a reflexive locally convex space, $h(x)$ be a measurable function and $x\rightarrow\omega_x\in\D^\times$ a weakly measurable map. Then the equation:
\begin{equation}
\label{RF}
\ip{f}{\omega_x}=h(x)
\end{equation}
admits a solution $f\in\D$ if, and only if, there exists a bounded subset $\mathcal M$ of $\D$ such that $|h(x)|\leq\sup_{f\in\mathcal M}|\ip{f}{\omega_x}|$  $\mu$-a.e.
\end{lemma}
\begin{proof}
Necessity is obvious. Conversely, let $x\in X$ be a point such that $\ip{f}{\omega_x}=h(x)\neq 0$.  Let us consider the subspace $V_x$ of $\D^\times$ given by $V_x:=\{\alpha\omega_x\}_{\alpha\in\mathbb C}$, and let us define the functional $\mu$ on $V_x$ by: $\mu(\alpha\omega_x):=\alpha h(x)$. We have that $|\mu(\alpha\omega_x)|=|\alpha h(x)|\leq |\alpha|\sup_{f\in\mathcal M}|\ip{f}{\omega_x}|=\sup_{f\in\mathcal M}|\ip{f}{\alpha\omega_x}|$, in other words: $|\mu(F_x)|\leq \sup_{f\in\mathcal M}|\ip{f}{F_x}|$ for all $F_x\in V_x$. By the Hahn-Banach theorem, there exists an extension $\tilde{\mu}$ to $\D^\times$ such that $|\tilde{\mu}(F)|\leq \sup_{f\in\mathcal M}|\ip{f}{F}|$, for every $F\in\D^\times$. Since $\D$ is reflexive, there exists $\bar{f}\in\D$ such that $\tilde{\mu}(F)=\ip{\bar{f}}{F}$. The statement follows from the fact that $\mu(\omega_x)=h(x)$.
\end{proof}
If $M$ is a subspace of $\D$ and the topology of $\D$ is generated by the family of seminorms $\{p_\alpha\}_{\alpha\in I}$, then the topology on the quotient space ${\D}/{ M}$ is defined, as usual, by the seminorms $\{\tilde{p}_\alpha\}_{\alpha\in I}$, where $\tilde{p}_\alpha(\tilde{f}):=\inf\{p_\alpha(g):g\in{f+M}\}$.
The following proposition can be compared to the case of Riesz-Fischer sequences: see \cite[Cap.4, Sec.2, Prop. 2]{young}.

\begin{prop}
\label{RF2}
Assume that $\D[t]$ is a Fr\'echet space. If the map $\omega: x\in X \to \omega_x\in \D^\times$ is a Riesz-Fischer map, then for every continuous seminorm $p$ on $\D$, there exists a constant $C>0$ such that, for every solution $f$ of \eqref{rf},
$$
\tilde{p}(\tilde{f}):=\inf\{p(g): g\in f+\omega^\bot\}\leq C\|\ip{f}{\omega_x}\|_2.
$$
\end{prop}
\begin{proof}
Since $\omega^\bot$ is closed, it follows that the quotient $\D/ \omega^\bot:=\D_{\omega^\bot}$ is a Fr\'echet space. If $f\in \D$, we put $\tilde{f}:=f+\omega^\bot$ . Let $h\in L^2(X,\mu)$ and $f$ a solution of (\ref{rf}) corresponding to $h$; then, we can define an operator $S: L^2(X,\mu)\rightarrow\D_{\omega^\bot}$ by $h\mapsto \tilde{f}$ . Let us consider a sequence $h_n\in L^2(X,\mu)$ such that $h_n\rightarrow 0$ and, for each $n \in {\mb N}$, let $f_n$ a corresponding solution of (\ref{rf}). One has that $\int_{X}|h_n(x)|^2d\mu\rightarrow 0$, i.e. $\int_{X}|\ip{f_n}{\omega_x}|^2d\mu\rightarrow 0$. This implies that $\ip{f_n}{\omega_x}\rightarrow 0$ in measure, so there exists a subsequence such that $\ip{f_{n_k}}{\omega_x}\rightarrow 0$ a.e. (see \cite{debarra}). On the other hand, if $Sh_n=\tilde{f}_n$ is a sequence convergent to $\tilde{f}$ in $\D_{\omega^\bot}$ w.r. to the quotient topology defined by the seminorms $\tilde{p}(\cdot)$, it follows that the sequence is convergent in \textit{weak} topology of $\D_{\omega^\bot}$, i.e.:
$$
\ip{\tilde{f}_n}{\tilde{F}}\rightarrow\ip{\tilde{f}}{\tilde{F}}\quad\forall \tilde{F}\in \D_{\omega^\bot}^\times.
$$
Let us consider the canonical surjection $\rho:\D\rightarrow\D_{\omega^\bot}$, $\rho:f\mapsto\tilde{f}=f+\omega^\bot$. Its transposed map (adjoint) $\rho^\dag:\D_{\omega^\bot}^\times\rightarrow\D^\times$ is injective (see \cite{horvath}, p. 263) and  $\rho^\dag[\D_{\omega^\bot}^\times]=\omega^{\bot\bot}$. Then $\rho^\dag:\D_{\omega^\bot}^\times\rightarrow\omega^{\bot\bot}$ is invertible. Hence,
$$
\ip{\tilde{f}_n}{\tilde{F}}=\ip{\rho({f_n})}{(\rho^\dag)^{-1}({F})}=\ip{{f_n}}{\rho^\dag((\rho^\dag)^{-1}({F}))}=\ip{{f_n}}{{F}}\quad\mbox{for all}\quad F\in\omega^{\bot\bot}.
$$
Thus, if $\tilde{f}_n\rightarrow\tilde{f}$ in the topology of $\D_{\omega^\bot}$, then $\ip{{f_n}}{{F}}\rightarrow\ip{{f}}{{F}}$, $\mbox{for all } F\in\omega^{\bot\bot}$, and, in particular, since $\omega\subset \omega^{\bot\bot}$, one has $\ip{f_{n}}{\omega_x}\rightarrow\ip{f}{\omega_x}$. Since $\ip{f_{n}}{\omega_x}$ has a subsequence convergent to $0$, one has $f\in\omega^\bot$. From the closed graph theorem, it follows that the map $S$ is continuous, i.e. for all continuous seminorms $\tilde{p}$ on $\D_{\omega^\bot}$ there exists $C>0$ such that: $\tilde{p}(Sh)\leq C\|h\|_2$, for all $h\in L^2(X,\mu)$. The statement follows from the definition of Riesz-Fischer map.
\end{proof}
\begin{cor}
\label{lbound}
Assume that $\D[t]$ is a Fr\'echet space. If the map $\omega: x\in X \to \omega_x\in \D^\times$ is a total Riesz-Fischer map, then for every continuous seminorm $p$ on $\D$, there exists a constant $C>0$ such that, for the solution $f$ of \eqref{rf},
$$
{p}({f})\leq C\|\ip{f}{\omega_x}\|_2
$$
\end{cor}
\berem
For an arbitrary weakly measurable map $x\rightarrow\omega_x\in\D^\times$, we define the following subset of $\D$: $D(V):=\{f\in\D: \ip{f}{\omega_x}\in L^2(\mathbb R)\}$ and the operator $Vf:=\ip{f}{\omega_x}$. Clearly,  $x\rightarrow\omega_x\in\D^\times$ is a Riesz-Fischer map if and only if $V:\D(V)\rightarrow L^2(\mathbb R)$ is surjective. If $\omega$ is total, it is injective too, so $V$ is invertible. A consequence of Corollary \ref{lbound} is that $V^{-1}:L^2(\mathbb R)\rightarrow \D(V)$ is continuous.
\enrem
\subsection{Duality}
\bedefi
Let $\D\hookrightarrow\H\hookrightarrow\D^\times$ be a rigged Hilbert space and $\omega$ a weakly measurable map. We call   {\em dual map of $\omega$}, if it exists, a weakly measurable map $\theta$ such that for all $f,g\in\D$:
$$\left|\int_{X}\ip{f}{\theta_x}\ip{\omega_x}{g}d\mu\right|<\infty$$ and
$$
\ip{f}{g}=\int_{X}\ip{f}{\theta_x}\ip{\omega_x}{g}d\mu, \quad \forall f,g\in\D.
$$
\findefi
\begin{prop} \label{rlf}
Suppose that $\omega$ is a  Riesz-Fischer map. Then  $\theta$ is a Bessel map.
\end{prop}
\begin{proof}
For all $h \in L^2(X,\mu)$ there exists $\bar{f}\in\D$ such that $\ip{\bar{f}}{\omega_x}=h(x)$ $\mu$-a.e. Since $\theta$ is a dual map, one has that: $|\int_X h(x)\ip{\theta_x}{g}d\mu|<\infty$ for all $h\in L^2(X,\mu)$. It follows that  $\ip{\theta_x}{g}\in L^2(X,\mu)$ (see \cite[Ch. 6, Ex. 4]{rudin}).
\end{proof}
\begin{prop}
Let $\D$ be reflexive and  let $\omega$ be a $\mu$-independent Bessel map.
Furthermore, suppose that for all $h\in L^2( X,\mu)$  there exists a bounded subset $\mathcal M\subset\D$ such that:
$$
\left|\int_{X} h(x)\ip{\omega_x}{g}d\mu\right|\leq \sup_{f\in\mathcal M}|\ip{f}{g}|,\quad\forall g\in\D,
$$
then the dual map $\theta$ is a Riesz-Fischer map.
\end{prop}
\begin{proof}
If $h\in L^2(X,\mu)$, and since $\omega$ is a Bessel map, one has: $|\int_{X} h(x)\ip{\omega_x}{g}d\mu|<\infty$. Let us consider $g=\int_{X}\ip{g}{\omega_x}\theta_xd\mu$ as element of $\D^\times$. We define the following functional on $\D$ (as subspace of $\D^\times$): $\mu(g):=\int_{X} h(x)\ip{\omega_x}{g}d\mu$. By hypothesis, one has:
$$
|\mu(g)|\leq\sup_{f\in\mathcal M}|\ip{f}{g}|.
$$
By the Hahn-Banach theorem, there exists an extension $\tilde{\mu}$ to $\D^\times$ such that:
$$\tilde{\mu}(G)\leq \sup_{f\in\mathcal M}|\ip{f}{G}|,\quad\forall G\in\D^\times.$$
Since $\D$ is reflexive, there exists
$\tilde{f}\in\D^{\times\times}=\D$ such that $\tilde{\mu}(G)=\ip{\tilde{f}}{G}$. In particular $\ip{\tilde{f}}{g}=\int_{X} h(x)\ip{\omega_x}{g}d\mu$. Since $\theta$ is dual of $\omega$, we have too: $\ip{\tilde{f}}{g}=\int_{X}\ip{\tilde{f}}{\theta_x}\ip{\omega_x}{g}d\mu$. Since $\omega$ is  $\mu$-independent, it follows that $h(x)=\ip{\tilde{f}}{\theta_x}$ $\mu$-a.e.
\end{proof}
\section{Semi-Frames and Frames}
\label{sect_4}
\subsection{Distribution Semi-Frames}
\bedefi
Given a rigged Hilbert space $\D\hookrightarrow\H\hookrightarrow\D^\times$, a Bessel  map $\omega$  is a {\it distribution upper semi-frame} if it is complete (total) and if there exists $B>0$:
\begin{equation}\label{eqn_bess_ex} 0<\int_X|\ip{f}{\omega_x}|^2d\mu \leq B \|f\|^2,\; \forall f\in \D, f\neq 0.
\end{equation}
\findefi
Since the injection  $\D\hookrightarrow\H$ is continuous, it follows that there exists a continuous seminorm $p$ on $\D$ such that $\|f\|\leq p(f)$ for all $f\in\D$.
If $\xi\in L^2(X,\mu)$, then the continuous conjugate functional $\Lambda_\omega^\xi$ on $\D$ defined in (\ref{eqn_lambda_om})
is bounded in $\D[\|\cdot\|]$; it follows that it has bounded extension $\tilde{\Lambda}_\omega^\xi$ to $\H$, defined, as usual, by a limiting procedure. Therefore, there exists a unique vector $h_\xi\in \H$ such that:
$$\tilde{\Lambda}_\omega^\xi (g)= \ip{h_\xi}{g}, \quad \forall g\in\H.$$
This implies that the synthesis operator $T_\omega$ takes values in $\H$, it is bounded and $\|T_\omega\|\leq B^{1/2}$; its hilbertian adjoint $D_\omega:=T_\omega^*$ extends the analysis operator $T_\omega^\times$.

 The action of $D_\omega$ can be easily described: if $g\in \H$ and  $\{g_n\}$ is a sequence of elements of $\D$, norm converging to $g$, then the sequence $\{\eta_n\}$, where $\eta_n(x)=\ip{g_n}{\omega_x}$, is convergent in $L^2(X, \mu)$. Put $\eta=\lim_{n\to \infty}\eta_n$. Then,
 \begin{equation}
\ip{T_\omega \xi}{g} =\lim_{n\to \infty}\int_X \xi(x)\ip{\omega_x}{g_n}d\mu=\int_X\xi(x)\overline{\eta(x)}d\mu.
\end{equation} Hence $T_\omega^*g=\eta.$

The function $\eta\in L^2(X,\mu)$ depends linearly on $g$, for each $x\in X$. Thus we can define a linear functional $\check{\omega}_x$ by
\begin{equation}\label{eqn_check}\ip{g}{\check{\omega}_x}= \lim_{n\to \infty}\ip{g_n}{\omega_x}, \quad g\in \H;\, g_n\to g.\end{equation}
Of course, for each $x\in X$, $\check{\omega}_x$ extends $\omega_x$;
however $\check{\omega}_x$ need not be continuous, as a functional on $\H$.
We conclude that:
$$T_\omega^*: g\mapsto \ip{g}{\check{\omega}_x}\in L^2(X, \mu)$$

Moreover, in this case, the sesquilinear form
$\Omega$ in \eqref{eqn_frameop}, which is well defined on $\D\times \D$, is bounded with respect to $\|\cdot\|$ and possesses a bounded extension $\hat{\Omega}$ to $\H$.
Hence there exists a bounded operator $\hat{S}_\omega$ in $\H$, such that
\begin{equation} \label{eqn_ext_Omega} \hat{\Omega}(f,g) =\ip{\hat{S}_\omega f}{g}, \quad \forall f,g \in \H. \end{equation}
Since
\begin{equation} \label{eqn_frameop_1}\ip{\hat{S}_\omega f}{g}=\int_X \ip{f}{\omega_x}\ip{\omega_x}{g} d\mu, \quad \forall f,g\in \D,
\end{equation}
 $\hat{S}_\omega$ extends the {frame operator} $S_\omega$ and $S_\omega:\D\to \H$. It is easily seen that
$\hat{S}_\omega = \hat{S}_\omega^*$ and $\hat{S}_\omega=T_\omega T_\omega^*$. By definition, we have:
$$
 0<\|\hat{S}_\omega f\| \leq B \|f\|,\; \forall f\in \H, f\neq 0
$$
Then $\hat{S}_\omega$ is bounded, self-adjoint and injective too. This means that ${\sf Ran}\,S_\omega$ is dense in $\H$, and $\hat{S}_\omega^{-1}$ is densely denfined, unbounded and self-adjoint (see \cite{semifr1}).
\berem
If $\{\omega_x\}_{x\in X}$ is an upper bounded semi-frame, then, for the continuity of injection $\D\hookrightarrow\H$, it is bounded by a continuous seminorm of $\D$ too. The vice-versa is not true: let us consider the rigged Hilbert space $\mathcal S(\mathbb R)\hookrightarrow L^2(\mathbb R)\hookrightarrow\mathcal S^\times(\mathbb R)$: the system of derivative of Dirac's deltas $\{\delta_x'\}_{x\in\mathbb R}$ is total. Since $\mathcal S(\mathbb R)$ is a Fr\`{e}chet space,  it holds  $(ii)$ of  Proposition \ref{prop2}. However $\{\delta_x'\}_{x\in\mathbb R}$ is not a distribution upper bounded semi-frame, in fact:
$$
\int_{\mathbb R}|\ip{\phi}{\delta'_x}|^2dx=\|\phi'\|_2^2\quad\forall\phi\in\mathcal S(\mathbb R)
$$
but the derivative operator  $\frac{d}{dx}: \mathcal S(\mathbb R)\rightarrow L^2(\mathbb R)$ is unbounded (clearly on the topology of Hilbert norm).
\enrem
\berem
In \cite{TTT} it is defined the notion of {\it bounded Bessel map}, that is a Bessel map in rigged Hilbert space such that: $\int_X|\ip{f}{\omega_x}|^2d\mu \leq B \|f\|^2,\; \forall f\in \D$. It is a more general notion than upper bounded semiframe. In fact, we can consider, as example,  the distribution $\omega_x:=\eta_K(x)\delta_x$ where $\eta_K(x)$ is a $C^\infty$-function with compact support $K$ and $M:=\max_{x\in K}|\eta_K(x)|$:
$$
\int_{\mathbb R}\!|\!\ip{\phi}{\omega_x}\!|^2dx\!=\!\!\int_{\mathbb R}|\!\ip{\phi}{{\eta_K}(x)\delta_x}\!|^2dx\!=\!\!\int_{\mathbb R}|{\overline{\eta_K(x)}\phi(x)}|^2dx\leq\! M^2\!\!\int_{K}\!|{\phi(x)}|^2dx\leq\! M^2\|\phi\|_2^2.
$$
Therefore $\omega$ is a bounded Bessel map, but it is not total, then it is not an upper bounded semi-frame.
\enrem
\bedefi
Given a rigged Hilbert space $\D\hookrightarrow\H\hookrightarrow\D^\times$, a  Bessel map $\omega$ is a {\it distribution lower Semi-Frame} if  there exists $A>0$ such that:
\begin{equation}\label{eqn_bess_ex} A \|f\|^2 \leq\int_X|\ip{f}{\omega_x}|^2d\mu ,\; \forall f\in \D, f\neq 0.
\end{equation}
\findefi
By definition, it follows that $\omega$ is total. One has that,  if $\D$ is a Fr\`{e}chet space, by Proposition \ref{prop2} it follows that the frame operator $S_\omega\in\LDD$ is unbounded.  Furthermore  $S_\omega$ is injective, and $S_\omega^{-1}$ is bounded.
\beex
Let us consider the space $\mathcal O_M$,   known (see \cite{reed1}) as the set of infinitely differentiable functions on ${\mathbb R}$ that are polynomially bounded together with their derivatives. Let us consider $g(x)\in\mathcal O_M$  such that $0<m<|g(x)|$. If we define $\omega_x:=g(x)\delta_x$, then $\{\omega_x\}_{x\in\mathbb R}$ is a distribution lower semi-frame with  $A=m^2$.
\enex
The  proof of the following Lemma is analogue to Lemma 2.5 in \cite{semifr1}:
\begin{lemma}
Let $\omega$ un upper bounded semi-frame with upper frame bound $M$ and $\theta$ a total family dual to $\omega$. Then $\theta$ is a lower semi-frame, with lower frame bound $M^{-1}$.
\end{lemma}
\subsection{Distribution Frames}
This section is devoted to frames, with main results already showed in \cite{TTT}.
{
\bedefi \label{defn_distribframe}  Let $\D[t]\subset\H\subset \D^\times[t^\times]$ be a rigged Hilbert space, with $\D[t]$ a reflexive space and $\omega$ a Bessel map.
We say that $\omega$ is a {\em  distribution frame} if there exist $A,B>0$ such that
\begin{equation} \label{eqn_frame_main1} A\|f\|^2 \leq \int_X|\ip{f}{\omega_x}|^2d\mu \leq B \|f\|^2, \quad \forall f\in \D. \end{equation}
\findefi
}
A distribution frame $\omega$ is clearly, in particular,  an upper bounded semi-frame. Thus, we can consider the operator $\hat{S}_\omega$ defined in \eqref{eqn_ext_Omega}. It is easily seen that, in this case,
$$ A\|f\| \leq \| \hat{S}_\omega f\| \leq B\|f\|,\quad \forall f\in \H. $$ This inequality, together with the fact that $\hat{S}_\omega$ is symmetric, implies that $\hat{S}_\omega$ has a bounded inverse $\hat{S}_\omega^{-1}$ everywhere defined in $\H$.

\berem
It is worth noticing that the fact that $\Omega$ and $S_\omega$ extend to $\H$ does not mean that $\omega$ a frame in the Hilbert space $\H$, because we do not know if the extension of ${S}_\omega$ has the form of \eqref{eqn_frameop} with $f,g \in \H$.
\enrem
To conclude this section, we reassume a list of properties of frames proved in \cite{TTT}.  We have the following:
\begin{lemma}\label{lemma_38}Let $\omega$ be a distribution frame. Then, there exists $R_\omega\in {\mc L}(\D)$ such that $S_\omega R_\omega f=R_\omega^\times S_\omega f= f$, for every $f\in \D$.
\end{lemma}
As consequence, the reconstruction formulas for distribution frames holds for all $f\in\D$:
$$
f=R_\omega^\times S_\omega f= \int_X\ip{f}{\omega_x}R_\omega^\times\omega_x d\mu;
$$

$$
f=S_\omega R_\omega f=\int_X\ip{R_\omega f}{\omega_x}\omega_x d\mu.
$$
These representations have to be interpreted in the weak sense.
\berem
The operator $R_\omega$ acts as an inverse of $S_\omega$. On the other hand the operator $\hat{S}_\omega$ has a bounded inverse $\hat{S}_\omega^{-1}$ everywhere defined in $\H$. It results  that  \cite[Remark 3.7]{TTT}: $\hat{S}_\omega^{-1}\D\subset\D$ and $R_\omega=\hat{S}_\omega^{-1}\upharpoonright_\D$.
\enrem
There exists the dual frame:
\begin{prop} \label{prop_dual2} Let $\omega$ be a distribution frame. Then there exists a weakly  measurable function $\theta$ such that:
$$ \ip{f}{g}= \int_X\ip{f}{\theta_x}\ip{\omega_x}{g}d\mu, \quad \forall f,g \in \D.$$
\end{prop}
Where $\theta_x:=R_\omega^\times\omega_x$
The frame operator $S_\theta$ for $\theta$ is well defined and we have:
$S_\theta= I_{\D,\D^\times}R_\omega$.

{The distribution function  $\theta$, constructed in Proposition \ref{prop_dual2}, is also a distribution frame, called the {\em canonical dual frame} of $\omega$}. Indeed, it results that \cite{TTT}:
$$ B^{-1}\|f\|^2 \leq \ip{S_\theta f}{ f} \leq A^{-1}\|f\|^2, \quad \forall f\in \D.$$

\subsection{Parseval distribution frames}
\bedefi
If $\omega$ is a distribution frame, then we say that:
\begin{itemize}
\item[a)]
$\omega$ is a {\it tight} distribution frame if we can choose A = B as frame bounds.
In this case, we usually refer to A as a frame bound for $\omega$;
\item[b)]
$\omega$ is a {\it Parseval} distribution frame if A = B = 1 are frame bounds.
\end{itemize}
\findefi
More explicitly  a weakly measurable distribution function $\omega$ is called a {\em Parseval distribution frame} if $$\int_X|\ip{f}{\omega_x}|^2d\mu = \|f\|^2, \quad f\in \D.$$
It is clear that a Parseval distribution frame is a frame in the sense of Definition \ref{defn_distribframe} with $S_\omega=I_\D$, the identity operator of $\D$. In \cite{TTT} is proved the following:
\begin{lemma} \label{lemma_00}Let  $\D\subset\H\subset \D^\times$ be a rigged Hilbert space and $\omega: x\in X \to \omega_x\in \D^\times$ a weakly measurable map. The following statements are equivalent.
\begin{itemize}
\item[(i)] $\omega$ is a Parseval distribution frame;
\item[(ii)]$\ip{f}{g}= \int_X \ip{f}{\omega_x}\ip{\omega_x}{g}d\mu, \quad \forall f, g \in \D$;
\item[(iii)]$f = \int_X \ip{f}{\omega_x}\omega_x d\mu$,  the integral on the r.h.s. is understood as a continuous conjugate linear functional on $\D$, that is an element of $\D^\times$.
\end{itemize}
\end{lemma}
The representation in (iii) of Lemma \ref{lemma_00} is not necessarily unique. 
\section{Distribution basis}
\label{sect_5}
\bedefi
 Let $\D[t]$ be a locally convex space, $\D^\times$ its conjugate dual and $\omega: x\in X\to \omega_x \in \D^\times$ a weakly measurable map. Then
$\omega$ is a {\it distribution basis} for $\D$ if, for every $f\in\D$, there exists a {\it unique} measurable function $\xi_f$ such that:
$$
\ip{f}{g}=\int_X\xi_f(x)\ip{\omega_x}{g}d\mu, \quad\forall f,g\in\D
$$
and, for every $x\in X$, the linear functional $f\in\D\rightarrow\xi_f(x)\in\mathbb C$ is continuous in $\D[t]$.
\findefi
\noindent
The distribution basis $\omega$  can be represented by:
$$
f=\int_X\xi_f(x)\omega_x d\mu
$$
 in weak sense.
\berem
Clearly, if $\omega$ is a distribution basis, then it is $\mu$-independent. Furthermore, since $f\in\D\rightarrow\xi_f(x)$ continuously, there exists a unique weakly $\mu$-measurable map $\theta:X\rightarrow \D^\times$ such that: $\xi_f(x)=\ip{f}{\theta_x}$ for every $f\in\D$. We call $\theta$ {\it dual} map of $\omega$. If $\theta$ is $\mu$-independent, then it is a distribution basis too.
\enrem
\subsection{Gel'fand distribution bases}
As mentioned in the introduction, Gel'fand \cite[Ch.4, Theorem 2]{gelf3} called {\em complete system} a $\D^\times$-valued function $\zeta$, which satisfies the Parseval equality and with the property that every $f\in \D$ can be uniquely written as $f=\int_X\ip{f}{\zeta_x}\zeta_x d\mu$, in weak sense.
As we shall see in the following discussion, these two conditions are a good substitute for the notion of {\em orthonormal basis} which is meaningless in the present framework.
\bedefi
A weakly measurable map $\zeta$ is {\em Gel'fand distribution basis} if it is a $\mu$-independent Parseval distribution frame.
\findefi
By definition and Lemma \ref{lemma_00}, this means that, for every $f\in \D$ there exists a unique function $\xi_f\in L^2(X, \mu)$ such that:
\begin{equation}
\label{bgelf}
 f=\int_X \xi_f(x)\zeta_x d\mu
\end{equation}
with $\xi_f(x)=\ip{f}{\zeta_x}$ $\mu$-a.e. Furthermore $\|f\|^2=\int_X |\ip{f}{\zeta_x}|^2d\mu$ and $\zeta$ is total too.

For every $x\in X$, the map $f\in \H\to \xi_f(x)\in {\mb C}$ defines as in \eqref{eqn_check} a linear functional $\check{\zeta_x}$ on $\H$, then for all $f\in \H$:
$$ f=\int_X \ip{f}{\check{\zeta_x}}{\zeta_x} d\mu.$$
We have the following characterization result \cite{TTT}:
\begin{prop}\label{prop_gelfandbasis}Let  $\D\subset\H\subset \D^\times$ be a rigged Hilbert space and let $\zeta: x\in X \to \zeta_x\in \D^\times$ be a Bessel distribution map. Then the following statements are equivalent.
\begin{itemize}
\item[(a)] $\zeta$ is a Gel'fand distribution basis.
\item[(b)] The synthesis operator $T_\zeta$ is an isometry of $L^2(X,\mu)$ onto $\H$.
\end{itemize}
\end{prop}
\beex \label{ex_3.15}
Given the rigged Hilbert space $\mathcal S(\mathbb R)\hookrightarrow L^2(\mathbb R)\hookrightarrow\mathcal S^\times(\mathbb R)$, for $ x\in\mathbb R$ the  function $\zeta_x(y)=\frac{1}{\sqrt{2\pi}}e^{-ixy}$, define a (regular) tempered distribution: in fact, denoting as usual by $\hat{g}$, $\check{g}$, respectively, the Fourier transform and the inverse Fourier transform of $g\in L^2(\RN)$, one has that $\mathcal S(\mathbb R)\ni\phi\mapsto\ip{\phi}{\zeta_x}=\frac{1}{\sqrt{2\pi}}\int_{\mathbb R}\phi(y)e^{-ixy}dy=\hat\phi(x)\in\mathbb C$.
For all $x\in\mathbb R$ the set of functions  $\zeta:=\{\zeta_x(y)\}_{x\in{\mathbb R}}$ is a Gel'fand distribution basis: in fact the synthesis operator $T_\zeta: L^2(\mathbb R)\rightarrow L^2(\mathbb R)$ defined by:
$$
(T_\zeta\xi)(x)=\frac{1}{\sqrt{2\pi}}\int_{\mathbb R}\xi(y)e^{-ixy}dy=\hat\xi(x)\quad\forall \xi\in L^2(\mathbb R)
$$
is an isometry onto $L^2(\mathbb R)$ by Plancherel theorem. The analysis operator is: $T_\zeta^*f=\check{f}, \quad \forall f \in L^2(\RN)$.

\enex
\beex \label{ex_3.16}
Let us consider $\mathcal S(\mathbb R)\hookrightarrow L^2(\mathbb R)\hookrightarrow\mathcal S^\times(\mathbb R)$ (again)
For $x\in\mathbb R$, let us consider the  Dirac delta $\delta_x: \mathcal S(\mathbb R)\rightarrow\mathbb C$, \, $\phi\mapsto \ip{\phi}{\delta_x}:=\phi(x)$.
The set of Dirac deltas $\delta:=\{\delta_x\}_{x\in\mathbb R}$ is a Gel'fand distribution basis. In fact, the Parseval identity holds:
$$
\int_{\mathbb R}|\ip{\delta_x}{\phi}|^2dx=\int_{\mathbb R}|\phi(x)|^2dx=\|\phi\|^2_2\quad\forall\phi\in\mathcal S(\mathbb R).
$$
The synthesis operator: $T_\delta: L^2(\mathbb R)\rightarrow L^2(\mathbb R)$ is:
$$\ip{T_\delta\xi}{\phi}=\int_{\mathbb R}\xi(x)\ip{\delta_x}{\phi}dx=\int_{\mathbb R}\xi(x)\overline{\phi(x)}dx=\ip{\xi}{\phi} \quad\forall\phi\in\mathcal S(\mathbb R)$$
then $T_\delta\xi=\xi$ for all $\xi\in L^2(\mathbb R)$. Since $T_\delta$ is an identity, it is an isometry onto $L^2(\mathbb R)$.
\enex

\medskip
\subsection{Riesz distribution basis}
Proposition \ref{prop_gelfandbasis} and  \eqref{bgelf} suggests a more general class of bases that will play the same role as Riesz bases in the ordinary Hilbert space framework.
\bedefi
Let  $\D\subset\H\subset \D^\times$ be a rigged Hilbert space. A weakly measurable map $\omega: x\in X \to \omega_x\in \D^\times$  is a {\em Riesz distribution basis} if $\omega$ is a $\mu$-independent distribution frame.
\findefi
One has the following \cite{TTT}:
\begin{prop}\label{prop_rieszbasis}Let  $\D\subset\H\subset \D^\times$ be a rigged Hilbert space and let $\omega: x\in X \to \omega_x\in \D^\times$ be a Bessel distribution map. Then the following statements are equivalent.
\begin{itemize}
\item[(a)]  $\omega$ is a Riesz distribution basis;
\item[(b)] If $\zeta$ is a Gel'fand distribution basis, then the operator $W$ defined, for $f\in \H$, by
$$ f=\int_X \xi_f(x)\zeta_x d\mu \to W f= \int_X \xi_f(x)\omega_x d\mu$$ is continuous and has bounded inverse.
\item[(c)] The synthesis operator $T_\omega$ is a topological isomorphism of $L^2(X, \mu)$ onto $\H$.
\end{itemize}
\end{prop}
\begin{prop} If $\omega$ is a Riesz distribution basis then $\omega$ possesses a unique dual frame $\theta$ which is also a Riesz distribution basis.
\end{prop}
\beex
Let us consider $f\in C^\infty(\mathbb R)$: $0<m<|f(x)|< M$. Let us define $\omega_x:=f(x)\delta_x$: then $\{\omega_x\}_{x\in\mathbb R}$ is a distribution frame:
$$
\int_{\mathbb R}|\ip{\omega_x}{\phi}|^2dx=\int_{\mathbb R}|\overline{f(x)}\phi(x)|^2dx\leq M^2\|\phi\|_2^2,\quad\forall\phi\in\mathcal S(\mathbb R).
$$
and
$$
m^2\|\phi\|_2^2 \leq\int_{\mathbb R}|\overline{f(x)}\phi(x)|^2dx \leq M^2\|\phi\|_2^2,\quad\forall\phi\in\mathcal S(\mathbb R)
$$
Furthermore, $\{\omega_x\}_{x\in\mathbb R}$ is $\mu$ independent, in fact, putting:
$$
\int_{\mathbb R}\xi(x)\ip{\omega_x}{g}dx=0,\quad\forall g\in\mathcal S(\mathbb R),
$$
one has:
$$
\int_{\mathbb R}\xi(x)\ip{\omega_x}{g}dx=\int_{\mathbb R}\xi(x){\overline{f(x)}\ip{\delta_x}{g}}dx=0,\quad \forall g\in\mathcal S(\mathbb R)
$$
since $\{\delta_x\}_{x\in\mathbb R}$ is $\mu$-independent, it follows that $\xi(x){\overline{f(x)}}=0$ q.o., then $\xi(x)=0$ q.o.
By definition, $\{\omega_x\}_{x\in\mathbb R}$ is a Riesz  distribution basis.
\enex
\section*{Concluding remarks}
In Hilbert space, frames, semi-frames, Bessel, Riesz-Fischer sequences,  and Riesz bases  are related through the action of a linear operator on elements of orthonormal basis  (see also \cite[Remark 3.22]{TTT}). On the other hand,  in literature have been already considered some studies on  bounds (upper and lower) of these sequences and their links with the linear operators related to them (see \cite{balastoeva}\cite{casachliin} and for semi-frames  \cite{semifr1}). For that, it is desirable to continue an analogue study in rigged Hilbert space and linear operators in $\LDD$.

\bibliographystyle{amsplain}

\end{document}